\documentclass[a4paper,10pt,reqno]{amsart}

\usepackage{preamble}
\newcommand\scalemath[2]{\scalebox{#1}{\mbox{\ensuremath{\displaystyle #2}}}}

\makeatletter
\def\l@subsection{\@tocline{2}{0pt}{4pc}{6pc}{}}
\makeatother

\title{Intersections of real symmetric hypersurfaces}

\author{Samuel Lidz} 
\address{Department of Mathematics, University of Maryland, College Park, MD 20742, USA}
\email{\href{mailto:slidz@umd.edu}{slidz@umd.edu}}
\thanks{}

\author{Zachary Lihn} 
\address{Department of Mathematics, Columbia University, New York, NY 10027, USA}
\email{\href{mailto:zal2111@columbia.edu}{zal2111@columbia.edu}}
\thanks{}

\author{Adam Melrod} 
\address{Department of Mathematics, University of Maryland, College Park, MD 20742, USA}
\email{\href{mailto:abmelrod@umd.edu}{abmelrod@umd.edu}}
\thanks{}

\date{August 2024}

\begin{document}

\begin{abstract}
     We prove a symmetric version of B\'ezout's theorem. More precisely, we show that the symmetric orbit type of a transverse intersection of complex symmetric hypersurfaces in projective space is determined by the degrees. In the projective plane, we fully classify the possible orbit types of such intersection loci using completely elementary methods. From this classification, we obtain strong restrictions on the number of real points in the intersection of real symmetric curves. We also provide a partial classification in $\P^3_\C$, with a similar restriction on real points.
\end{abstract}

\maketitle

\renewcommand\contentsname{\vspace{-1cm}}
\tableofcontents

\setlength{\parindent}{0em}
\parskip 0.75em

\section{Introduction}

The classical B\'ezout's theorem states that the intersection of $n$ hypersurfaces $V(f_1, \dots, f_n) \subseteq \P_\C^n$ consists of $\deg f_1 \cdots \deg f_n$ points, counted with multiplicity. If the hypersurfaces are invariant under a group action, one expects to obtain enriched, \emph{equivariant} information about the orbits of the induced action on the vanishing set. More generally, one expects a theory of equivariant enumerative geometry calculating enriched counts 
of solutions to an enumerative problem with a group action. 

In this paper, we apply the results of \cite{brazelton2024equivariant} to obtain a symmetric version of B\'ezout's theorem for symmetric polynomials over the complex numbers. Assuming the corresponding hypersurfaces are transverse, the intersection locus gains the structure of an $S_{n+1}$-set, where $S_{n+1}$ is the symmetric group on $(n+1)$ letters. We show in Section 2 that the resulting $S_{n+1}$-orbit types depend only on the degrees of the chosen hypersurfaces, i.e. are independent of the specific hypersurfaces chosen. Following similar methods as \cite{brazelton2024equivariant}, we describe how to apply our theorem to obtain bounds on the number of real points in the intersection locus of real symmetric hypersurfaces. 

In Section 3, we specialize to the case of $\P^2_\C$, with $S_3$ acting by permuting the coordinates. We give a full classification of the possible orbit types of the intersection locus, depending only on the degrees. We also find conditions under which intersections of symmetric curves are never transverse. More precisely, we prove the following.
\begin{thm}[Theorem \ref{p2intersection}]
    Let $C$ and $D$ be transverse symmetric curves in $\P^2_\C$ of degree $d$ and $e$, respectively. Then the symmetric orbit type of their intersection locus is determined by the product of the degrees modulo 6. It can be read off of the following table.
    \begin{center}
        \begin{tabular}{|c||c|c|c|c|c|c|}\hhline{|-||-|-|-|-|-|-|}
            $d \cdot e$ & $6k$ & $6k+1$ & $6k+2$ & $6k+3$ & $6k+4$ & $6k+5$ \\ \hhline{|-||-|-|-|-|-|-|}
            $S_3$-orbit type & $k[S_3]$ & \textcolor{red}{$\times$} & $k[S_3] + [S_3/C_3]$ & $k[S_3] +  [S_3/C_2]$ & \textcolor{red}{$\times$} & $k[S_3] +  [S_3/C_2] + [S_3/C_3]$ \\ \hhline{|-||-|-|-|-|-|-|}
        \end{tabular}
    \end{center}
    The symbol \textcolor{red}{$\times$} signifies that a transverse intersection with these degrees is not possible. 
\end{thm}
 Using Theorem \ref{p2intersection}, we also obtain the following theorem on the size of the real intersection locus of transverse real symmetric curves. 
 
\begin{thm}[Theorem \ref{p2realpoints}]
    Let $C$ and $D$ be real transverse symmetric curves in $\P^2_\C$ of degrees $d$ and $e$ respectively, with corresponding polynomials $f$ and $g$. 
    \begin{enumerate}[(a)]
        \item If $de \equiv 3,5\bmod{6}$, then $V(f,g)$ has $3+6k$ real points for some integer $0\leq k < de/6$.
        \item If $de \equiv 0,2\bmod{6}$, then $V(f,g)$ has $6k$ real points for some integer $0\leq k\leq de/6$.
    \end{enumerate}
\end{thm}

In Section 4, we carry out a similar analysis in $\P^3_\C$ with $S_4$ acting. Here, we only obtain a partial classication of the orbits. However, we can still use this to make strong statements about the real intersection locus of transverse real symmetric surfaces, and about conditions for non-transversality. For example, we prove the following.
\begin{thm}[Theorem \ref{p3realpoints}]
    If any three real symmetric surfaces in $\P^3_\C$ intersect transversely, the number of real points in their intersection is a multiple of 12.
\end{thm}
 It is clear that our methods can be used to obtain similar results in higher dimensions, although their effectiveness decreases  due to the rapidly growing size and complexity of the symmetric group.

Related work includes an equivariant version of B\'ezout's theorem for the conjugation action shown in \cite{C2equivar} using different methods to obtain information on real points of generic hypersurfaces. Enriched equivariant counts of different enumerative problems were also carried out in \cite{brazelton2024equivariant} and \cite{bethea2024enrichedcountnodalorbits}.

Our results, particularly Theorems \ref{p2intersection} and \ref{p2realpoints}, point towards the possibility of a more general enumerative theory counting the real loci of hypersurface intersections. We expect an equivariant characteristic class computation that agrees with our results. This paper provides a first step towards establishing such a theory and provides concrete examples for future work.


\subsection*{Notation} We adopt standard notation for writing isomorphism classes of $S_{n}$-sets as elements in the Burnside ring of $S_n$, although we will not use this technology further. More specifically, we write the isomorphism class of an $S_n$-set as a disjoint union of isomorphism classes of $S_n$-orbits, using $+$ to represent disjoint union. Isomorphism classes of orbits are written as $[S_n/H]$, where $H$ is the stabilizer (isotropy) subgroup of a point in the orbit ($H$ is determined up to conjugation). For example, 
\[
2[S_3/C_2] + [S_3/C_3]
\]
represents a disjoint union of two orbits isomorphic to $S_3/C_2$ and one orbit isomorphic to $S_3/C_3$, for a total of eight elements.

 Throughout, we use $\P^n$ to denote $\P_\C^n$. 

\subsection*{Acknowledgements}
This project was carried out while the authors were participants of the Institute for Advanced Study/Park City Mathematics Institute 2024 Summer School in Motivic Homotopy Theory. The authors thank Thomas Brazelton and Candance Bethea for insightful discussions. They also thank J.D. Quigley and Jackson Morris for coordinating the PCMI Undergraduate Experimental Math Lab where this project began.



\section{Transverse intersections of symmetric hypersurfaces}

We first recall the main theorem of \cite{brazelton2024equivariant}, which they call \emph{equivariant conservation of number}.

\begin{thm}[\cite{brazelton2024equivariant}, {Equivariant Conservation of Number}] \label{thm:brazelton-equivariant-conservation}
    Let $E \to M$ be an equivariant complex rank $n$ vector bundle over a smooth $G$-manifold of dimension $n$, and let $\sigma,\sigma' \colon M \to E$ be any two sections whose zeros are isolated and simple. Then $Z(\sigma)$ and $Z(\sigma')$ are isomorphic as finite $G$-sets. In other words, the $G$-orbits of the zeros are independent of the choice of section.
\end{thm}

\begin{defn}
    A section $\sigma \colon M \to E$ has \emph{isolated} zeros if $Z(\sigma) \subseteq M$ is discrete. The zeros of $\sigma$ are \emph{simple} if the Jacobian determinant of $\sigma$ does not vanish at any point of $Z(\sigma)$.
\end{defn}



We now consider projective space $\P^n$ and let $d_1,\ldots, d_n\geq 1$ be integers. We may rephrase the classical B\'ezout's theorem as follows.

\begin{thm}[B\'ezout's Theorem]
    Let $f_1\in H^0(\P^n, \cO_{\P^n}(d_1)),\ldots, f_n \in H^0(\P^n, \cO_{\P^n}(d_n))$ be sections in general position. Then the section $(f_1,\ldots, f_n) \in H^0(\P^n, \oplus_{i=1}^n \cO_{\P^n}(d_i))$ has $d_1\ldots d_n$ zeroes, all of multiplicity 1.
\end{thm}

We generalize this to the symmetric setting. To that end, let $S_{n+1}$ act on $\P^n$ by permuting coordinates. This acts linearly, so extends to a compatible group action on all of the line bundles $\cO(k)$.
From Theorem \ref{thm:brazelton-equivariant-conservation} and B\'ezout's theorem we immediately obtain the following. 

\begin{thm}[Symmetric B\'ezout's Theorem]
    Let $f_1,\ldots, f_n$ be $S_{n+1}$-equivariant homogeneous polynomials in general position on $\P^n$ of degrees $d_1,\ldots, d_n$ respectively. Then their intersection $V(f_1,\ldots, f_n)$ is a finite $S_{n+1}$-set of points of size $d_1\cdots d_n$. Moreover, the $S_{n+1}$-isomorphism type of $V(f_1,\ldots, f_n)$ depends only on $d_1,\ldots, d_n$ and is independent of the generic section we chose. In other words, every intersection locus of transverse symmetric hypersurfaces of fixed degrees has the same $S_{n+1}$-isomorphism type.
\end{thm}

We note that the symmetric group action and conjugation action commute. This means that the orbit of any real point consists entirely of real points. We use this later to constrain the number of real points in the transverse intersection of real symmetric hypersurfaces.

\section{The case of $\P^2$}

In this section, we classify the possible $S_3$-orbit types of the intersection locus of two transverse symmetric curves in $\P^2$ and examine the consequences. Our main strategy for performing this classification is to analyze the possible points of intersection based on their isotropy subgroups. The possible nontrivial isotropy subgroups of $S_3$ are
\[S_3, C_3, C_2.\]

\subsection{Points on transverse intersections}

We begin by considering the fixed points under the whole $S_3$ action. Let $[a:b:c]$ be such a point. Since one of $a,b,c$ are nonzero, then all of them are nonzero, so we can write the point as $[a : b : 1]$. Then $[a : b : 1] = [b : a : 1]$ so $a = b$. Thus the point is $[a : a : 1]$. Further, we have $[a : a : 1] = [a : 1 : a]$ so $a = 1$, and thus the only fixed point under $S_3$ is $[1:1:1]$.
\begin{lem}\label{ones}
    If $f$ is a symmetric polynomial that vanishes on $[1:\cdots:1]\in \P^{n-1}$, then it has a singularity at $[1:\cdots :1]$. Note that this is the only $S_n$-fixed point in $\P^{n-1}$, so there are no $S_n$-fixed points in transverse intersections.
\end{lem}
\begin{proof}
    Recall Euler's lemma for homogeneous polynomials:
    \begin{equation*}
        \frac{\partial f}{\partial X_1} X_1 + \cdots + \frac{\partial f}{\partial X_n} X_n = \deg f \cdot f.
    \end{equation*}
    Since $f$ is symmetric, 
    \[
    \frac{\partial f}{\partial X_i}(1,\dots, 1) = \frac{\partial f}{\partial X_j}(1.\dots,1)
    \]
    for any $i,j$. Thus, for any $i$, Euler's lemma gives
    \[
    n \frac{\partial f}{\partial X_i}(1,\ldots, 1) = \deg f \cdot f(1,\ldots, 1).
    \]
    Hence, if $f$ vanishes at $[1:\cdots:1]$, then all its partial derivatives vanish there too and $f$ is singular at $[1:\cdots:1]$.
\end{proof}

 From this, we deduce that $[1:1:1]$ can never be in the intersection locus.

Next we consider fixed points $[a:b:c]$ under $C_2$, which is generated by a transposition. Without loss of generality, we may assume the transposition swaps the first two coordinates, so $[a:b:c] = [b:a:c]$. First, if $c = 0$, then the point assumes the form $[a:1:0] = [1/a:1:0]$. Hence in this case $a = 1/a$ so $a^2 = 1$ and thus $a = \pm 1$. Second, if $c \neq 0$, then the point is in the form $[a:b:1] = [b:a:1]$ in which case $a = b$. It follows that all points with isotropy group $C_2$ are (in the $S_3$-orbit of) one of the following points
    \[[-1:1:0], [1:1:0], [a:a:1].\]

 The following two lemmas show that of these, only $[-1:1:0]$ can appear in the intersection of transverse symmetric curves.

\begin{lem}{\label{diagonal-line}}
    Fix $f$ a symmetric homogeneous polynomial with corresponding curve $V(f)$ in $\P^2$. If $f$ has a simple zero at $P = [a:a:1]$, then the tangent line to $V(f)$ at $P$ is given by $X + Y - 2aZ=0$. In particular, two symmetric curves intersecting at $P$ are not transverse.
\end{lem}
\begin{proof} 
    We consider the dehomegenization of $f$. As $P$ is a simple point, we have the following formula (\cite{fulton2008algebraic} p. 33) for the tangent line to $P$ on $V(f)$.
    \begin{align*}
        \frac{\partial f}{\partial x}(P)(x-a) + \frac{\partial f}{\partial y}(P)(y-a) &= 0.
    \end{align*}
    Note that $\frac{\partial f}{\partial x}(P) = \frac{\partial f}{\partial y}(P)$ since $f$ is symmetric. The derivatives are nonzero, since we assumed that $V(f)$ is nonsingular at $P$. Hence the equation of the line is equivalent to $(x-a) + (y-a) = 0$, so we conclude. 

    For the latter statement, if either of the symmetric curves has a multiple zero at $P$ then we are done. If they both have a simple zero at $P$, we showed they have the same tangent line at $P$. Hence their intersection at $P$ has multiplicity at least 2.
\end{proof}

\begin{lem}
    Let $f$ be a symmetric homogeneous polynomial. If $f$ has a simple zero at the point $[1:1:0]$, then the tangent line to $V(f)$ at $[1:1:0]$ is given by $Z=0$.
    In particular, two symmetric curves intersecting at $[1:1:0]$ cannot be transverse.
\end{lem}
\begin{proof}
     Let us assume $f$ has a simple zero at $[1:1:0]$. Setting $Z=0$, applying Euler's lemma in two variables, and using the symmetry of $f$, we have 
    \[
        \frac{\partial f}{\partial X} (1,1,0) = \frac{\partial f}{\partial Y} (1,1,0) = 0.
    \] The tangent line to $V(f)$ at $[1:1:0]$ is given by
    \[
        \frac{\partial f}{\partial X}(1,1,0) X + \frac{\partial f}{\partial Y}(1,1,0)Y +\frac{\partial f}{\partial Z}(1,1,0) Z = 0
    \] which is just
    \[
        \frac{\partial f}{\partial Z}(1,1,0) Z = 0.
    \]
    

     In particular, for two symmetric curves passing through $[1:1:0]$, either at least one is singular there, or they share a tangent line, so the intersection multiplicity is at least 2.
\end{proof}

It remains to consider the $C_3$-invariant points. The $C_3$ subgroup is generated by a 3-cycle, which corresponds to cycling the coordinates. The coordinates of a fixed point must all be non-zero, so the point can be expressed as $[a:b:1]$. Applying the action, we get the conditions $a=\frac 1 b = \frac b a$. Hence, $a^3=1$, $a=\frac 1 b$, $b^3=1$. Letting $\omega$ be a primitive cube root of unity, we see that the only possible points are $[1:1:1], [\omega : \omega^2: 1], [\omega^2 : \omega : 1]$. The first was ruled out by Lemma \ref{ones}.

Thus, we have determined the points with a given isotropy subgroup that may appear in an intersection of transverse symmetric curves. 
We summarize the results here:

\begin{table}[h]
    \begin{center}
    \bgroup
    \def\arraystretch{1.2}
    \begin{tabular}{|c|c|} \hline
        Isotropy Subgroup & Fixed Points \\ \hline\hline
        \text{Trivial} & all points
         \\\hline
        $C_2$ & $[-1 : 1 : 0]$ \\\hline
         $C_3$ & $[\omega : \omega^2 : 1]$ \\ \hline
        $S_3$ & none \\ \hline
    \end{tabular}
    \egroup
    \caption{Possible fixed points in transverse intersection in $\P^2$.}
    \label{table:P2-fixedpoints}
    \end{center}
\end{table}
 We list only one point in each $S_3$-orbit; for example $[\omega:\omega^2:1]$ is listed, so $[\omega^2, \omega, 1]$ is omitted. Any point with a given isotropy subgroup must be in the $S_3$-orbit of the point we provide.

\subsection{Consequences}

Table \ref{table:P2-fixedpoints} immediately implies the following corollary.

\begin{cor} \label{cor:uniqueness-of-certain-orbit-types}
    The intersection of any two transverse symmetric curves cannot have more than one orbit of type $[S_3/C_3]$, and cannot have more than one orbit of type $[S_3/C_2]$.
\end{cor}

\begin{rmk}
    We have classified the points in an $[S_3/C_3]$-orbit as $[\omega:\omega^2:1]$ and $[\omega^2, \omega, 1]$. As a result, we know that an $[S_3/C_3]$-orbit must be complex and the conjugation action exchanges the two points in the orbit. 
    A similar observation holds true for an $[S_3/C_2]$-orbit, which must consist of the three real points $[-1:1:0], [-1:0:1]$, and $[0:-1:1]$.
\end{rmk}

\begin{thm} \label{thm:P2-nontransverse-1mod3}
    Let $C$ and $D$ be symmetric curves in $\P^2$ of degrees $d$ and $e$ respectively. If $d e \equiv 1 \bmod{3}$, then their intersection is not transverse.  
\end{thm}

\begin{proof}
    Suppose the intersection is transverse. The number of points in the $[S_3]$ and $[S_3/C_2]$ orbits is $0 \bmod{3}$. Since there is at most one orbit of type $[S_3/C_3]$ and none of type $[S_3/S_3]$, the total number of points must be 0 or 2 mod 3, a contradiction.
\end{proof}

\begin{thm}{\label{p2intersection}}
    Let $C$ and $D$ be transverse symmetric curves in $\P^2$ of degree $d$ and $e$, respectively. Then the symmetric orbit type of their intersection locus is determined by the product of the degrees modulo 6. It can be read off of the following table.
    \begin{center}
        \begin{tabular}{|c||c|c|c|c|c|c|}\hhline{|-||-|-|-|-|-|-|}
            $d \cdot e$ & $6k$ & $6k+1$ & $6k+2$ & $6k+3$ & $6k+4$ & $6k+5$ \\ \hhline{|-||-|-|-|-|-|-|}
            $S_3$-orbit type & $k[S_3]$ & \textcolor{red}{$\times$} & $k[S_3] + [S_3/C_3]$ & $k[S_3] +  [S_3/C_2]$ & \textcolor{red}{$\times$} & $k[S_3] +  [S_3/C_2] + [S_3/C_3]$ \\ \hhline{|-||-|-|-|-|-|-|}
        \end{tabular}
    \end{center}
    The symbol \textcolor{red}{$\times$} signifies that a transverse intersection with these degrees is not possible. 
\end{thm}
\begin{proof}
    This follows from Corollary \ref{cor:uniqueness-of-certain-orbit-types}, Theorem \ref{thm:P2-nontransverse-1mod3}, and the sizes of the possible orbits in Table \ref{table:P2-fixedpoints} modulo 6.
\end{proof}

Note that this gives an independent proof of Symmetric B\'ezout's theorem for $\P^2$. From the theorem we also obtain constraints on the counts of real points on intersections of transverse symmetric real curves in $\P^2$.
\begin{thm}{\label{p2realpoints}}
    Let $C$ and $D$ be real transverse symmetric curves in $\P^2$ of degrees $d$ and $e$ respectively, with corresponding polynomials $f$ and $g$. 
    \begin{enumerate}[(a)]
        \item If $de \equiv 3,5\bmod{6}$, then $V(f,g)$ has $3+6k$ real points for some integer $0\leq k < de/6$.
        \item If $de \equiv 0,2\bmod{6}$, then $V(f,g)$ has $6k$ real points for some integer $0\leq k\leq de/6$.
    \end{enumerate}
\end{thm}
\begin{proof}
    In the situation of (a) we know there is a single $[S_3/C_2]$-orbit in $V(f,g)$. By our classification, the points in this orbit are real. Therefore we have at least $3$ real points. The $[S_3/C_3]$-orbit is complex (again by our classification). Also, the points in an $[S_3]$-orbit are either all real or all complex. We deduce (a). Similarly (b) is follows from  Theorem \ref{p2intersection}.
\end{proof}

\begin{ex}
    Let us compute the orbit type of a specific example. We take the intersection $V(X+Y+Z,X^5+Y^5+Z^5)$ of degree 1 and 5 symmetric curves. One can verify that the intersection consists of 5 points.
    The five points are:
    \[\scalemath{1}{\left[\frac{1}{2} \left(-1-i \sqrt{3}\right) : \frac{1}{2} \left(-1+i \sqrt{3}\right) : 1 \right], \quad \left[\frac{1}{2} \left(-1+i \sqrt{3}\right) : \frac{1}{2} \left(-1-i \sqrt{3}\right) : 1 \right]},\]
    \[[-1:0:1], \quad \scalemath{1}{[0:-1:1]}, \quad \scalemath{1}{[1:-1:0]}.\]
    This has orbit type $ [S_3/C_3] + [S_3/C_2]$ and 3 real points, in agreement with Theorems \ref{p2intersection} and \ref{p2realpoints}.
\end{ex}

\section{Partial classification of orbit types in $\mathbb{P}^3$}

\begin{lem}{\label{repeat-coordinates}}
    Let $P=[a:b:c:d]\in \P^3$. If any two of $a,b,c$, or $d$ are equal, then any three symmetric polynomials $f_1,f_2,f_3$ vanishing at $P$ have intersection multiplicity greater than one at $P$.
\end{lem}
\begin{proof}
    Without loss of generality we may assume $a=b$ and $d=1$, so we have the point $P = [a,a,c,1]$. Dehomogenize with respect to the last variable to obtain symmetric polynomials $f_1,f_2,f_3$ with variables $x_0,x_1,x_2$. Since the $f_i$ are symmetric, we obtain
    \[
    \frac{\partial f_i}{\partial x_0}(a,a,c) = \frac{\partial f_i}{\partial x_1} (a,a,c)
    \]
    for all $i$. Let $a_{i,j} = \frac{\partial f_i}{\partial x_j} (a,a,c)$, so the Jacobian matrix at $P$ is 
    \[
    \begin{pmatrix}
        a_{1,0} & a_{1,0} & a_{1,2} \\
        a_{2,0} & a_{2,0} & a_{2,2} \\
        a_{3,0} & a_{3,0} & a_{3,2} \\
    \end{pmatrix}
    \]
    which has rank at most 2. Thus the intersection multiplicity must be greater than 1 at the point.
\end{proof}

    There are 11 conjugacy classes of subgroups in $S_4$, which we describe below.
    For each conjugacy class of subgroups of $S_4$, we list the forms that fixed points can take (up to permutation of coordinates, or equivalently up to choice of subgroup in the conjugacy class). We mark points that cannot appear in a transverse intersection 
    in \singular{red} (using Lemma \ref{repeat-coordinates}).
\begin{enumerate}
    \item The trivial subgroup. Everything is fixed.
    \item $C_2^o=\langle (1 \;2)\rangle$, the odd conjugacy class of cylic subgroups of order 2.
    \begin{itemize}
        \item    Fixed points: \singular {$[a : a : c : 1]$, $[a : a : 1 : 0]$ or $[\pm 1 : 1 : 0 : 0]$}.
    
    \begin{proof}
        Invariance under this action means \[[a : b :c :d]= [b :a :c : d].\] If $d$ is nonzero, then forces $b = a$. If $d = 0$ and $c$ is nonzero, then similarly $b = a$. Otherwise, both $a$ and $b$ must be nonzero. Normalizing $b$ to be $1$, we have that $a = 1/a$ and so $a = \pm 1$.
    \end{proof}
    \end{itemize}
    
    \item $C_2^e = \langle (1\;2) (3\;4)\rangle$, the even conjugacy class of cyclic subgroups of order 2.
    \begin{itemize}
        \item Fixed points: $[a : - a : - 1 : 1]$, $\singular{[a : a : 1 : 1]}$, $\singular{[- 1 : 1 : 0 : 0]}$, and $\singular{[1 : 1 : a : a]}$.
        \begin{proof}
            Let $[a:b:c:d]$ be a fixed point. Either both $a$ and $b$ are nonzero or both $c$ and $d$ are nonzero. Since the cases are symmetric, we analyze only the latter case. Then the point maybe be normalized to be $[a:b:c:1]$, and being fixed implies that $[a:b:c:1] = [b/c : a/c : 1/c : 1]$. This implies that $c = \pm 1$, so $a=b/c$ means $b = \pm a$.
        \end{proof}
    \end{itemize}    

    \item $C_4= \langle (1\;2\;3\;4)\rangle$, the cyclic subgroup of order 4.

    \begin{itemize}
        \item Fixed points: $\singular{[1:1:1:1], [-1:1:-1:1]}, [i: -1: -i:1], [-i:-1:i:1]$.

        \begin{proof}
            Being fixed under $C_4$ implies all coordinates are nonzero, and imposes the following constraints 
            \[
            [a : b : c : 1] = [1/c : a/c : b/c : 1] = [c/b : 1/b : a/b : 1] = [b/a : c/a : 1/a : 1].
            \] 
            We can see that $a = 1/c = c/b = b/a$ implies 
            $c = 1/a$ and thus $b = 1/a^2$, so $a = b/a = 1/a^3$,and  $a^4 = 1$. Thus the points are of the form
            $[a : a^2 : a^3 : 1]$ where $a^4 = 1$.
        \end{proof}
    \end{itemize}
    
    \item $K_4^{\triangleleft} = \{ (),(1\;2)(3\;4), (1\;3)(2\;4), (1\;4)(2\;3) \}$, the normal Klein 4-subgroup of $S_4$.

    \begin{itemize}
        \item Fixed points: \singular{$[1 : -1 : -1 :1], [-1 : 1 : -1 : 1], [1 : 1 : 1 : 1], [-1 : -1 : 1 : 1]$}.
        
        \begin{proof}
            First, such invariant points are fixed under $(12)(34)$, which are of the form $[a : - a : - 1 : 1]$, $[a : a : 1 : 1]$, $[- 1 : 1 : a : - a]$, and $[1 : 1 : a : a]$.

            Suppose it is $[a : - a : - 1 : 1]$. Invariance under $(13)(24)$ implies $[a : - a : - 1 : 1] = [-1 : 1 : a : -a]= [1/a : -1/a : -1 : 1]$. Hence $a=1/a$ so $a = \pm 1$.

            In all the other cases, we similarly deduce that $a=\pm1$.
        \end{proof}
    \end{itemize}

    \item $K_4=\langle (1\;2), (3\;4)\rangle $, the non-normal Klein 4-subgroup of $S_4$.

    \begin{itemize}
        \item Fixed points: \singular{$[0 : 0 : -1 : 1], [a : a :1 : 1], [-1 : 1 : 0 : 0], [1 : 1 : a : a]$}.
        \begin{proof}
            Note that $C_2^e = \langle (12)(34)\rangle $ is a subgroup, so by the prior classification the fixed points are of the form $[a : - a : - 1 : 1]$, $[a : a : 1 : 1]$, $[- 1 : 1 : a : - a]$, and $[1 : 1 : a : a]$. Moreover, these are invariant under permuting the first and second coordinate, and the third and fourth coordinate. This forces the potential options to be $[0 : 0 : -1 : 1], [a : a :1 : 1], [-1 : 1 : 0 : 0], [1 : 1 : a : a]$.
        \end{proof}
    \end{itemize}
    
    \item $D_8 = \langle (1\;2\;3\;4), (1\;3) \rangle$, the dihedral subgroup with size 8.
    \begin{itemize}
        \item Fixed points: \singular{$[1:1:1:1], [-1:1:-1:1]$}
        
        \begin{proof}
            This is the case of $C_4$ with the additional constraint that $a=c$, which gives $a^2=1$.
        \end{proof}
    \end{itemize}
    
    \item $C_3=A_3 =\langle (1\;2\;3)\rangle = \{ (), (1\;2\;3), (1\;3\;2) \}$, the cyclic subgroup of order 3, viewed as the alternating group on 3 letters embedded in $S_4$ with size 3.
    \begin{itemize}
        \item Fixed points: $\singular{[a:a:a:1]}, \singular{[1:1:1:0]}, [\omega:\omega^2:1:0], [\omega^2:\omega:1:0]$. Note that the first two families of points can also be written as $\singular{[0:0:0:1], [1:1:1:a]}$.
        \begin{proof}
            If $d\neq 0$, we can normalize so $d=1$, and then the $C_3$-invariance implies $a=b=c$.

            Now suppose $d=0$. If one of the $a,b,c$ equals zero then all of them do, so we suppose all are nonzero. Therefore we normalize $c=1$. The $C_3$-invariance gives the equation
            \[
            [a:b:1:0] = [1:a:b:0] = [1/b: a/b :1:0]
            \]
            so that $a=\frac{1}{b}, b=\frac{a}{b}$. This then implies $b^3= 1,b^2=a$. Thus we obtain $b\in \{1,\omega,\omega^2\}$ where $\omega$ is a primitive third root of unity, and then we obtain the corresponding $a$. 
        \end{proof}
    \end{itemize}
    
    \item $S_3 = \langle (1\;2\;3), (1\;2)\rangle$, the symmetric group on 3 letters.
    \begin{itemize}
        \item Fixed points: \singular{$[1 : 1 : 1 : 0]$ and $[a : a : a : 1]$}. 
        \begin{proof}
            Let the fixed point be $[a:b:c:d]$. The point is $A_3$-invariant, and of the possible choices there, the only ones there which are not $S_3$-invariant are the ones involving primitive cube roots of unity.
        \end{proof}
    \end{itemize}
   
    \item $A_4 = \langle (1\;2\;3), (1\;2)(3\;4)\rangle$, the alternating group with size 12. 
    \begin{itemize}
        \item Fixed points: $\singular{[1:1:1:1]}$.
        \begin{proof}
             $A_4$ contains $A_3$, and the $A_3$-fixed points $[1:1:1:0], [\omega:\omega^2:1:0], [\omega^2:\omega:1:0]$ are not invariant under $(12)(34)$. It thus suffices to consider points of the form $[a:a:a:1]$. The invariance under $(12)(34)$ means that $[a:a:a:1] = [a:a:1:a]$ so $a = 1$.
        \end{proof}
    \end{itemize}
    
    \item $S_4$, the whole group.
    \begin{itemize}
        \item Fixed points: $\singular{[1:1:1:1]}$.
    \end{itemize}
\end{enumerate}

The results of the computations are summarized in Table \hyperref[table:P3-fixedpoints]{2}, which lists for each conjugacy class of $S_4$ the fixed points that could lie on a transverse intersection of three symmetric surfaces in $\P^3$. As in Table \ref{table:P2-fixedpoints}, we list only one point in each $S_4$-orbit, corresponding to a choice of representative in the conjugacy class. Here $a\in \C$ is a free parameter.
\begin{table}[h]
    \begin{center}
    \bgroup
    \def\arraystretch{1.15}
    \begin{tabular}{|c|c|} \hline
        \text{Isotropy subgroup} & \text{Fixed points} \\ \hline \hline
        \text{Trivial} & \text{all points} \\ \hline
        $C_2^e$ & $[a:-a:-1:1]$ \\ \hline
        $C_3$ & $[\omega : \omega^2 : 1 : 0]$ \\ \hline
        $C_4$ & $[i : -1 : -i : 1]$ \\ \hline
        \phantom{a}$C_2^o, K_4^{\triangleleft}, K_4, S_3, D_8, A_4, S_4$ \phantom{a}  & \text{none} \\ \hline
    \end{tabular}
    \egroup
    \label{table:P3-fixedpoints}
    \caption{Possible fixed points in transverse intersection in $\P^3$.}
    \end{center}
\end{table}

Although we cannot fully determine the orbit type from the product of the degrees, we can still make strong statements about the nonexistence of transverse intersections. We are also able to restrict the number of real points in the intersection.

\begin{thm}
    If three symmetric surfaces in $\P^3$ intersect transversely, then the product of their degrees is congruent to $0,2,6$ or 8 modulo 12.
    
\end{thm}
\begin{proof}
    There is at most one $C_3$-invariant orbit, and at most one $C_4$-invariant orbit (see Table \hyperref[table:P3-fixedpoints]{2}.). These orbits have size 8 and 6 respectively. The other possible orbits have size 12 and 24, so the total number of points, which is the product of the degrees, must be $12k$, $12k+6$, $12k+8$, or $12k+14$ (depending on whether there is a $C_3$-invariant orbit and/or a $C_4$-invariant orbit).
\end{proof}

\begin{thm}{\label{p3realpoints}}
    If  three real symmetric surfaces in $\P^3$ intersect transversely, the number of real points in their intersection is a multiple of 12.
\end{thm}
\begin{proof}
    The only subgroups $H$ of $S_4$ for which real points can be $H$-invariant are the trivial subgroup and $C_2^e$. The size of the orbit of such a point must be 24 or 12.
\end{proof}


\bibliographystyle{amsalpha} 
\bibliography{bibliography}

\end{document}